\documentclass[12pt]{article}

\usepackage[leqno]{amsmath}
\usepackage{amsthm,amscd}
\usepackage{amsfonts} 
\usepackage{amssymb} 

\newtheorem{theorem}{Theorem}[section]

\newtheorem{co}[theorem]{Corollary}

\newtheorem{rem}[theorem]{Remark}
\newtheorem{example}[theorem]{Example}

\def\diag{\mathop{\rm diag}}
\def\tr{\mathop{\rm tr}}
\def\rank{\mathop{\rm rank}}

\def\kmms{\kern-\mathsurround}

\newcommand{\calX}{{\cal X}}

\newcommand{\calZ}{{\cal Z}}

\newcommand{\bbR}{{\mathbb R}}
\newcommand{\bbS}{{\mathbb S}}

\begin{document}

\title{A unified approach to marginal equivalence 
in the general framework of group invariance}
\author{
Hidehiko Kamiya\footnote
{
{\rm This work was partially supported by 
JSPS 
KAKENHI Grant Number 25400201. } 
} 
\\ 
{\it Graduate School of Economics, Nagoya University}
}
\date{March 2014}

\maketitle

\begin{abstract}
Two Bayesian models with different sampling densities 
are said to be marginally equivalent 
if the joint distribution of observables 
and the parameter of interest is the same 
for both models. 
We discuss marginal equivalence 
in the general framework of group invariance. 
We introduce a class of sampling models 
and establish marginal equivalence 
when the prior for the nuisance parameter is relatively invariant. 
We also obtain some robustness properties of 
invariant statistics under our sampling models. 
Besides the prototypical example of $v$-spherical distributions, 
we apply our general results to 
two examples---analysis of affine shapes and principal component analysis. 

\smallskip
\noindent
{\it Key words}: 
Bayesian model, 
group invariance, 
marginal equivalence, 
null robustness, 
relatively invariant measure, 
$v$-spherical distribution. 

\smallskip
\noindent
{\it MSC2010}: 62H10, 62E15, 62F15. 
\end{abstract}


\section{Introduction} 
Two Bayesian models with different sampling densities 
are said to be marginally equivalent 
(Osiewalski and Steel \cite{OS93JE59}, 
Fern\'{a}ndez, Osiewalski and Steel \cite{FOS95JASA}) 
if the joint distribution of observables 
and the parameter of interest is the same 
for both models. 
In the present paper, we discuss marginal equivalence 
in the general framework of group invariance. 

For elliptical distributions with the scale matrix 
known up to a scalar multiplication, 
suppose this scalar is a nuisance parameter and that 
the parameter of interest is the location. 
For a fixed scale matrix up to a scalar, 
different density generators 
give rise to different elliptical distributions.  
When the prior for the nuisance scalar parameter 
is taken to be the noninformative prior, 
Bayesian models with one of these elliptical sampling densities 
are all marginally equivalent (Osiewalski and Steel \cite{OS93JE57}). 
A similar result is shown to hold true for $l_q$-spherical distributions 
(Osiewalski and Steel \cite{OS93Biometrika}). 

In an excellent paper \cite{FOS95JASA}, 
Fern\'{a}ndez, Osiewalski and Steel introduced 
a wide class of flexible multivariate distributions called 
$v$-spherical distributions. 
Essentially the same class of distributions 
(without location and scale) 
was studied independently 
by Kamiya, Takemura and Kuriki \cite{KTK08} 
under the name ``star-shaped distributions.'' 
Furthermore, 
the class of $v$-spherical distributions is as rich as  
the $D$-class of distributions of Ferreira and Steel \cite{FS05JRSS(B)}. 
The class of $v$-spherical distributions  
includes elliptical distributions and $l_q$-spherical distributions, 
but it also allows asymmetry of distributions 
and includes, e.g., multivariate skewed exponential power distributions.  
Fern\'{a}ndez, Osiewalski and Steel \cite{FOS95JASA} 
proved that marginal equivalence remains to hold true 
for this wide class of $v$-spherical distributions. 

The motivation of the present paper 
is to extend the discussion of marginal equivalence for 
$v$-spherical distributions 
in \cite{FOS95JASA} 
to the general situation of group invariance. 
We examine marginal equivalence 
under a class of sampling models 
based on the orbital decomposition of the sample space. 
We also study robustness of the distributions of invariant statistics 
under our sampling models.
These sampling models can be considered a generalization of 
the $D$-class of distributions 
(or the class of $v$-spherical distributions). 
By making discussions in the general framework, 
we can apply our results to problems other than those 
about multivariate distributions with location and scale. 

The organization of this paper is as follows. 
We introduce our sampling models 
in Section \ref{sec:sampling-models}. 
Next, in Section \ref{sec:r-of-i.s.} 
we obtain some results about robustness of invariant statistics, 
and in Section \ref{sec:ME} we establish marginal equivalence 
when the prior for the nuisance parameter is relatively invariant. 
The prototypical example of our arguments is, of course, 
the $v$-spherical distribution. 
In Sections \ref{sec:sampling-models}--\ref{sec:ME}, 
we illustrate our discussions 
with the case of the $v$-spherical distribution. 
In Section \ref{sec:examples}, we apply our general results to 
two other examples---analysis of 
affine shapes in Subsection \ref{sec:affine-shapes} 
and principal component analysis (PCA) in Subsection \ref{sec:PCA}.  
We conclude with some concluding remarks in 
Section \ref{sec:concluding-remarks}. 

\section{Sampling models}
\label{sec:sampling-models}

In this section, we introduce our sampling models in the general 
framework of group invariance. 
For basic concepts in group invariance in statistics,
the reader is referred to Section 1.4 of Kariya and Kurata \cite{KK04}. 

Suppose a group $H$ acts on 
a space $\calX$ to the left: 
$H\times \calX \ni (h,x) \mapsto hx \in \calX$.  
Moreover, suppose another group $G$ acts on a 
subspace $\calX_* \subseteq \calX$ 
to the left: 
$G \times \calX_* \ni(g,x)\mapsto gx \in \calX_*$.  
Let $r: \calX_* \to G$ be a map which is equivariant: 
\begin{equation} 
\label{eq:rgx=grx} 
r(gx)=gr(x), \quad g\in G, \ x\in \calX_*. 
\end{equation} 
Define 
\begin{equation}  
\label{eq:def-of-z(x)} 
z(x):=\{ r(x)\}^{-1}x, 
\quad x \in \calX_*, 
\end{equation}  
and put 
\[ 
\calZ:= 
\{ z(x): x \in \calX_* \}. 
\]    
By the definition of $z( \, \cdot \, )$,  
we can write every $x\in \calX_*$ as 
$x=r(x)z(x)$. 
Furthermore, we can easily see that $\calZ$ is a cross section 
under the action of $G$ on $\calX_*$, i.e., 
$\#(\calZ \cap Gx)=1$ for each $x\in \calX_*$, 
where $Gx=\{ gx: g\in G\}$ is the orbit of $x$.   
Thus, when we write $x\in \calX_*$ as $x=gz, \ g\in G, \ z\in \calZ$, 
the point $z$ is unique, i.e., $z=z(x)$.  
We can also see that the existence of $r( \, \cdot \, )$ 
satisfying \eqref{eq:rgx=grx} 
implies the action of $G$ on $\calX_*$ is free so that 
for every $x \in \calX_*$, 
the element $g$ in $x=gz, \ g\in G, \ z\in \calZ$, is 
unique as well, i.e., $g=g(x)$. 
Hence $\calX_*$ is in one-to-one correspondence 
with $G\times \calZ$ 
(an orbital decomposition of $\calX_*$). 
Note that $z=z(x)$ is a maximal invariant under the action 
of $G$ on $\calX_*$. 

We assume 
i) $\calX$ is a locally compact Hausdorff space, 
ii) $\calX_*$ is a locally compact Hausdorff space 
(e.g., $\calX_*$ is open in $\calX$), 
iii) $H$ and $G$ are second countable, locally compact Hausdorff 
topological groups and 
iv) the action of $H$ (resp. $G$) on $\calX$ (resp. $\calX_*$) is continuous.  

Let $\lambda$ be a 
relatively invariant measure on  
$\calX$ under the action of $H$ 
with multiplier $\chi_H$: 
$\lambda(d(hx))=\chi_H(h)\lambda(dx), \ h \in H$. 
Suppose $\lambda$ 
restricted to 
$\calX_*$ 
is relatively invariant under the action of $G$ 
with multiplier $\chi_G$: 
$\lambda(d(gx))=\chi_G(g)\lambda(dx), \ g \in G$. 
We assume 
\begin{equation} 
\label{eq:l(X-hX*)=0}
\lambda(\calX \setminus h\calX_*)=0
\end{equation} 
for each $h\in H$, 
where $h\calX_*=\{ hx: x \in \calX_*\}$.  

Now, as a sampling model, 
we consider the distribution on $\calX_*$ 
with density, with respect to $\lambda$, of the form 
\begin{equation} 
\label{eq:p(x|hgf)} 
p(x | h, g; f)
=\frac{1}{ \chi_H(h) \chi_G(g) }
f(g^{-1}r(h^{-1}x)s(z(h^{-1}x))), 
\quad 
h \in H, \ g \in G
\end{equation} 
for some $s: \calZ \to G$ 
and $f: G \to \bbR_{\ge 0}:=\{ u \in \bbR: u\ge 0\}$ 
such that 
\begin{equation}
\label{eq:int-f-l=1} 
\int_{\calX_*} f(r(x)s(z(x)))\lambda(dx)=1. 
\end{equation}  
It may happen that $h^{-1}x \notin \calX_*$ for 
$x \in \calX_*$ and  $h \in H$, 
but this does not matter because of \eqref{eq:l(X-hX*)=0}. 
We assume there exist some compact subgroup 
$H_0$ of $H$ and some compact subgroup $G_0$ of $G$ 
such that 
\begin{gather}
f(g_0 g)=f(g) \text{ \ for all \ } g \in G, \ g_0 \in G_0, 
\notag 
\\ 
r(h_0x)=r(x) \text{ \ and \ } s(z(h_0x))=s(z(x)) 
\text{ \ for all \ } x\in \calX_*, \ h_0\in H_0 
\label{eq:r(h0x)=r(x)-s(z(h0x))=s(z(x))}
\end{gather} 
(again recall \eqref{eq:l(X-hX*)=0}). 
Thus, $p( \, \cdot \, | h, g; f)=p( \, \cdot \, | h', g'; f)$ 
if $([h], [g])=([h'], [g'])$, 
where $[h]:=hH_0$ and $[g]:=gG_0$ are the left cosets 
modulo $H_0$ and $G_0$, respectively. 
Note $\chi_H(h_0)=1$ for $h_0 \in H_0$ and 
$\chi_G(g_0)=1$ for $g_0\in G_0$ 
because $H_0$ and $G_0$ are compact subgroups 
(Wijsman \cite{Wij90}, Corollary 7.1.8). 
By abuse of notation, we will write \eqref{eq:p(x|hgf)} also as 
$p(x | [h], [g]; f)$.  

Furthermore, we assume the converse so that 
$p( \, \cdot \, | [h], [g]; f)$ and $p( \, \cdot \, | [h'], [g']; f)$ give the 
same distribution if and only if $([h], [g])=([h'], [g'])$. 
In other words, 
we will consider only those 
$f$s 
such that this equivalence holds true. 
Therefore, 
\[
\Theta:=H/H_0 \times G/G_0
\]  
can be regarded as 
the parameter space, where $H/H_0=\{ [h]: h \in H \}$ and 
$G/G_0=\{ [g]: g \in G\}$ are the left coset spaces.  
Suppose $[h]$ is the parameter of interest while $[g]$ is a 
nuisance parameter. 

We illustrate our discussions with the prototypical example of 
the $v$-spherical distribution. 

\begin{example} 
\label{ex:v-setup}

An $n$-dimensional random vector $x$ is said to 
be distributed as 
a $v$-spherical distribution 
(Fern\'{a}ndez, Osiewalski and Steel \cite{FOS95JASA}) 
with location $\mu\in \bbR^n$ and scale 
$\sigma \in \bbR_{>0}:=\{ g\in \bbR: g>0 \}$ 
if $x$ has density, with respect to Lebesgue 
measure, of the form 
\begin{equation}
\label{eq:v-density}
\frac{1}{\sigma^n}
f(v(\sigma^{-1}(x-\mu ))), 
\end{equation} 
where 
$v: \bbR^n\setminus \{ 0\}\to \bbR_{>0}$ 
satisfies 
\[ 
v(gx)=gv(x) \text{ \ \ for \ } x \in \bbR^n\setminus \{ 0\}, \ g \in \bbR_{>0}. 
\] 
Note that the $v$-spherical distribution with 
$v(x)=(x^T \Sigma_0^{-1}x)^{1/2}$ for some 
$\Sigma_0 \in {\rm PD}_n$ 
(the set of $n\times n$ positive definite matrices) 
is an elliptical distribution. 
Now we will see that $v$-spherical distributions can be regarded as 
\eqref{eq:p(x|hgf)}. 

Let  
$\calX=\bbR^n$, 
$\calX_* =\bbR^n\setminus \{ 0\}$, 
$H=\bbR^n$ 
(regarded as the additive group under vector addition) 
and 
$G=\bbR_{>0}$ 
(the multiplicative group of positive reals). 
(We make no notational distinction between a group 
and its underlying set.) 
Consider the actions 
\[ 
H\times \calX \ni (\mu, x)\mapsto x+\mu \in \calX 
\]  
and 
\[ 
G\times \calX_* \ni (g, x)\mapsto gx \in \calX_* 
\text{ (scalar multiplication)}. 
\] 
Let  
$\lambda(dx)=dx$ (Lebesgue measure).  
Then 
$\chi_H(\mu)=1$ and 
$\chi_G(g)=g^{n}$. 

Let us take 
$r(x)=v(x), \ x \in \calX_*$. 
In that case,  
\[ 
\calZ=\left\{ z(x): x \in \bbR^n\setminus \{0\}\right\} 
\text{ with } z(x)= \frac{1}{v(x)} x 
\] 
is the boundary of the star-shaped set 
(with respect to the origin) determined by $v( \, \cdot \, )$. 
Let 
$s( \, \cdot \, )=1$. 
Then 
$p(x | \mu, \sigma; f), 
\ \mu \in H, \ \sigma \in G$, 
in \eqref{eq:p(x|hgf)} is  
\eqref{eq:v-density}.  

In the case of the $v$-spherical distribution, 
$H_0$ and $G_0$ are trivial. 
Examples with non-trivial $H_0$ or $G_0$ will be 
given in Section \ref{sec:examples}. 

\end{example} 

\section{Robustness of invariant statistics} 
\label{sec:r-of-i.s.}

In this section, we will see the null robustness of the distribution of 
$z(x)$ when $x$ is distributed according to our sampling model  
\eqref{eq:p(x|hgf)}. 
Namely, the distribution of $z(x)$ will be shown not to depend on 
the density generator $f$ when the parameter of interest, $[h]$, 
is the coset containing the identity element. 
Moreover, we will prove that the distribution of any statistic 
invariant under the actions of $G$ and $H$ does not depend on 
$f$ under general values of $[h]$. 

It is known 
(Farrell \cite[Theorem 10.1.2]{Far85}, 
Wijsman \cite[Theorem 7.5.1]{Wij90}) 
that under weak regularity conditions,  
there exists a measure $\nu_{\calZ}$ on $\calZ$ such that 
\begin{equation}
\label{eq:lambda-decomposition}
\lambda = \chi_G \mu_G \otimes \nu_{\calZ}
\end{equation} 
with respect to the one-to-one correspondence 
\[ 
\calX_* \ni x \leftrightarrow (r,z)\in G \times \calZ, 
\] 
where $\mu_G$ is a left invariant measure on $G$, 
which is unique up to a multiplicative constant.   
(In the expression \eqref{eq:lambda-decomposition}, we are identifying 
$\calX_*$ with $G \times \calZ$.) 
We assume throughout 
that the decomposition \eqref{eq:lambda-decomposition} holds true. 
Note that $\nu_{\calZ}$ appears in the decomposition of 
$\lambda$ and hence does not depend on $f$.  

Using \eqref{eq:lambda-decomposition},
we can prove the following theorem. 
We denote the identity element of $H$ by $e_H$.  

\begin{theorem} 
\label{th:null-robust}
Fix $r: \calX_* \to G$ satisfying \eqref{eq:rgx=grx}, 
and $s: \calZ \to G$. 
For $x \sim p(x | [h], [g]; f)\lambda(dx)$, 
the distribution of the statistic  
$z=z(x)$ is null robust in the sense that when 
$[h]=[e_H]=H_0$, 
the distribution of $z$ 
does not depend on $f$. 
Specifically,  when $[h]=[e_H]$, 
the distribution of $z$ has density 
$p_{\calZ}(z | [e_H], [g]; f)$ 
with respect to $\nu_{\calZ}$ given by 
\[
p_{\calZ}(z | [e_H], [g]; f)
=p_{\calZ}(z):=\frac{c}{\chi_G(s(z))\Delta_G(s(z))}, 
\]
where $\Delta_G$ is the right-hand modulus of $G$ 
(i.e., $\mu_G(d(gg'))=\Delta_G(g')\mu_G(dg), \ g' \in G$) 
and 
\begin{equation} 
\label{eq:c=intG=intZ} 
c=\int_G f(r)\chi_G(r)\mu_G(dr)
=\frac{1}{\int_{\calZ} \{ \chi_G(s(z))\Delta_G(s(z))\}^{-1} \nu_{\calZ}(dz)}.  
\end{equation}  
\end{theorem} 

\begin{proof} 
Since $[h]=[e_H]$, we have 
\begin{eqnarray} 
\label{eq:(1/chi)fchimunu} 
x 
&\sim& \frac{1}{ \chi_G(g) } f(g^{-1}r(x)s(z(x)))\lambda(dx)  \nonumber \\ 
&=& \frac{1}{ \chi_G(g) } f(g^{-1}rs(z)) \chi_G(r)\mu_G(dr)\nu_{\calZ}(dz). 
\end{eqnarray}
Putting $r'=g^{-1}rs(z)$, we can write \eqref{eq:(1/chi)fchimunu} as 
\begin{gather*} 
\frac{1}{\chi_G(g)}f(r')\chi_G(gr'(s(z))^{-1})\Delta_G((s(z))^{-1})\mu_G(dr') \nu_{\calZ}(dz) \\ 
= f(r')\chi_G(r')\mu_G(dr') \times \frac{1}{\chi_G(s(z))\Delta_G(s(z))}\nu_{\calZ}(dz). 
\end{gather*} 
Hence, the density of $z=z(x)$ with respect to $\nu_{\calZ}$ is 
\[
p_{\calZ}(z | [e_H], [g]; f)=\frac{c}{\chi_G(s(z))\Delta_G(s(z))}
\] 
with $c$ given in 
\eqref{eq:c=intG=intZ}. 
\end{proof} 

We note that $p_{\calZ}(z | [e_H], [g]; f)=p_{\calZ}(z)$ does not depend 
on $[g]$ either. 
Denote the identity element of $G$ by $e_G$. 
When $[g]$ in $p(x | [h], [g]; f)$ 
is restricted to be $[e_G]$, 
Theorem \ref{th:null-robust} is essentially the same as 
Theorem 3.3 of Kamiya, Takemura and Kuriki \cite{KTK08}. 

\begin{example} 
\label{ex:v-null}
{\bf (Continuing Example \ref{ex:v-setup}.)} 
Since $G=\bbR_{>0}$ is commutative, we have 
$\Delta_G=1$.  

When $\mu=0$, 
we have from Theorem \ref{th:null-robust} that 
$p_{\calZ}(z | 0, \sigma; f)=p_{\calZ}(z)=c$, 
so $z \sim c \, \nu_{\calZ}(dz)$. 
For $\mu_G(dg)$, we will take the version 
$\mu_G(dg)
=g^{-1}dg$, 
so $\nu_{\calZ}$ is such that 
\[ 
dx
=g^n \cdot \frac{dg}{g} \cdot \nu_{\calZ}(dz)
=g^{n-1}dg \, \nu_{\calZ}(dz)
\] 
for $x=gz$, 
and 
$c=\int_{0}^{\infty}f(r)r^{n-1}dr=1/\nu_{\calZ}(\calZ)$.  

It is known (Kamiya, Takemura and Kuriki \cite{KTK08}) 
that when $v( \, \cdot \, )$ is piecewise 
of class $C^1$, 
the measure $\nu_{\calZ}$ is given by 
\[ 
\nu_{\calZ}(dz)=\langle z, n_z \rangle \lambda_{\calZ}(dz), 
\] 
where $\lambda_{\calZ}(dz)$ is the volume element of $\calZ$,  
the vector $n_z$ is the outward unit normal vector of $\calZ$ at $z$ 
and   
$\langle \, \cdot \, , \, \cdot \, \rangle$ denotes the standard 
inner product of $\bbR^n$. 
So in that case, 
\[
z \sim c \, \langle z, n_z \rangle \lambda_{\calZ}(dz), 
\quad 
c=\int_0^{\infty}f(r)r^{n-1}dr
=\frac{1}{\nu_{\calZ}(\calZ)}
=\frac{1}{\int_{\calZ}\langle z, n_z \rangle \lambda_{\calZ}(dz)}. 
\]

Let us consider the elliptical case: 
$v(x)=(x^T \Sigma_0^{-1}x)^{1/2}$.  
We have 
$n_z=(1/\| \Sigma_0^{-1}z \|)\Sigma_0^{-1}z$, 
so  
$\langle z, n_z \rangle = (z^T\Sigma_0^{-2}z)^{-1/2}$ 
(notice $v(z)=1$ for $z \in \calZ$). 
Since 
\[
c=\frac{1}{\int_{\calZ}(z^T\Sigma_0^{-2}z)^{-1/2}\lambda_{\calZ}(dz)}
=\int_0^{\infty}f(r)r^{n-1}dr
\]  
is independent of the choice of $f$, 
we can obtain the value of $c$ by considering 
the particular case of normality as 
\[ 
c=\frac{1}{\omega_n (\det \Sigma_0)^{1/2}}, 
\] 
where $\omega_n=2\pi^{n/2}/\Gamma(n/2)$ is the total volume of 
$\bbS^{n-1}$ (the unit sphere in $\bbR^n$). 
Thus 
\[ 
z\sim \frac{1}{\omega_n (\det \Sigma_0)^{1/2}(z^T\Sigma_0^{-2}z)^{1/2}}
\lambda_{\calZ}(dz) 
\]
in this case. 
See Section 4 of Kamiya, Takemura and Kuriki \cite{KTK08} 
for details. 

\end{example} 

In Examples \ref{ex:v-setup} and \ref{ex:v-null}, 
we considered the distribution of $z(x)=\{ v(x)\}^{-1}x$ when 
$x$ has a $v$-spherical distribution (with $\mu=0$). 
Theorem \ref{th:null-robust} can also 
yield the distribution of the usual direction of $x$ 
when $x$ has 
a $v$-spherical distribution: 

\begin{example} 
\label{ex:v-direction} 
{\bf (Continuing Example \ref{ex:v-setup}.)} 
In the situation of Example \ref{ex:v-setup}, 
let us take, instead, 
\begin{gather*} 
r(x)
= \| x \|, \quad x \in \calX_*, \\  
s(z)
= v(z), \quad z \in \calZ=\bbS^{n-1}. 
\end{gather*} 
In that case, 
\[ 
r(x) s(z(x))=\| x\| \, v\left( \frac{1}{\| x\|}x\right) =v(x), 
\] 
so 
$p(x | \mu, \sigma; f), \  
\mu \in H, \ \sigma \in G$, in 
\eqref{eq:p(x|hgf)} 
in this case is also 
\eqref{eq:v-density}. 
(See also Ferreira and Steel \cite{FS05JRSS(B)}.) 
But now 
Theorem \ref{th:null-robust} gives the 
distribution of the usual direction $z(x)=\| x \|^{-1}x \in \bbS^{n-1}$ 
when 
$x$ has a $v$-spherical distribution with $\mu=0$: 
since $p_{\calZ}(z)=c \, \{ v(z)\}^{-n}$ and 
\[ 
\nu_{\bbS^{n-1}}(dz)
=\langle z, z \rangle \lambda_{\bbS^{n-1}}(dz) = \lambda_{\bbS^{n-1}}(dz), 
\]  
we have 
\[ 
z=z(x)=\frac{1}{ \| x\|}x \sim c \, \{ v(z)\}^{-n}\lambda_{\bbS^{n-1}}(dz), 
\quad 
c=\int_0^{\infty}f(r)r^{n-1}dr
=\frac{1}{\int_{\bbS^{n-1}}
\{ v(z) \}^{-n}
\lambda_{\bbS^{n-1}}(dz)}. 
\] 
In the elliptical case: $v(x)=(x^T \Sigma_0^{-1}x)^{1/2}$, 
we have 
\begin{equation} 
\label{eq:dist-dir-under-elliptical} 
z=z(x)=\frac{1}{ \| x\|}x 
\sim c \, 
(z^T \Sigma_0^{-1}z)^{-n/2}
\lambda_{\bbS^{n-1}}(dz)
\end{equation}  
with 
\[
c=\frac{1}{\int_{\bbS^{n-1}}
(z^T \Sigma_0^{-1}z)^{-n/2}
\lambda_{\bbS^{n-1}}(dz)}
=\int_0^{\infty}f(r)r^{n-1}dr
=\frac{1}{\omega_n (\det \Sigma_0)^{1/2}}. 
\] 
The distribution in \eqref{eq:dist-dir-under-elliptical} is 
noted in King \cite[p.1266]{King80}. 
See also the derivation of this distribution 
in Watson \cite[pp.109--110]{Wat83}. 

\end{example} 

Since $z=z(x)$ is a maximal invariant under the action of $G$, 
any invariant statistic $t=t(x), \ x\in \calX_*$,  
under the action of $G$ 
(i.e., $t(gx)=t(x), \ g \in G, \ x \in \calX_*$) is 
a function of $z=z(x)$, 
and thus the distribution of $t$ is also null robust.  
Moreover, when $t$ is invariant under the action of $H$ as well, 
$t(hx)=t(x), \ h \in H, \ x \in \calX_*$ 
(recall \eqref{eq:l(X-hX*)=0}),  
the robustness of the distribution of $t$ is valid for general 
$[h] \in H/H_0$:  

\begin{co} 
\label{co:inv-under-GandH} 
Fix $r: \calX_* \to G$ satisfying \eqref{eq:rgx=grx} 
and $s: \calZ \to G$, and 
assume $x \sim p(x | [h], [g]; f) \lambda(dx)$. 
Suppose a statistic $t=t(x)$ 
is invariant under the action of $H$ 
as well as under the action of $G$. 
Then the distribution of $t$ does not depend on $f$. 
\end{co} 

\begin{proof} 
The statistic $t=t(x)$, being invariant under the action of $G$, 
depends on $x$ only through $z(x)$: 
$t=t(x)=\tilde{t}(z(x))$ for some $\tilde{t}: \calZ \to t(\calX_*)$. 
For an arbitrary $\tilde{h}\in [h]$, let $x_0:=\tilde{h}^{-1}x$. 

Because of the invariance of  $t(x)$ under the action of $H$, 
we can write 
\begin{equation} 
\label{eq:t=tzx0}  
t=t(x)
=t(x_0)=\tilde{t}(z(x_0)). 
\end{equation}  

Now, write $\tilde{h}=hh_0, \ h_0\in H_0$. 
Since $H_0$ is compact, we have $\chi_H(h_0)=1$. 
Using this and \eqref{eq:r(h0x)=r(x)-s(z(h0x))=s(z(x))}, 
we obtain from \eqref{eq:p(x|hgf)} that 
\begin{eqnarray} 
\label{eq:dist-of-x0} 
x_0 \sim \frac{1}{\chi_G(g)}f(g^{-1}r(x_0)s(z(x_0)))\lambda(dx_0) 
= p(x_0 | [e_H], [g]; f)\lambda(dx_0). 
\end{eqnarray} 
Note that this distribution of $x_0=\tilde{h}^{-1}x$ 
does not depend on the choice of $\tilde{h}\in [h]$. 
By \eqref{eq:dist-of-x0} and Theorem \ref{th:null-robust}, 
the distribution of $z(x_0)$ does not depend on $f$. 

This fact, together with \eqref{eq:t=tzx0}, implies that  
the distribution of $t$ does not depend on $f$.  
\end{proof} 

In Example \ref{ex:v-setup}, 
the action of $H$ on $\calX$ is 
transitive, so statistics invariant under the action of $H$ 
are constants. 
An example of Corollary \ref{co:inv-under-GandH} is given by 
the normalized residual vector 
in the linear regression model with a $v$-spherical error distribution: 

\begin{example} 
\label{ex:v-regression} 
{\bf (Continuing Example \ref{ex:v-setup}.)} 
Let us consider the linear regression model 
with a $v$-spherical error distribution: 
\begin{gather*}
y = X\beta + \epsilon, 
\quad \epsilon \sim \frac{1}{\sigma^n}f(v(\sigma^{-1}\epsilon))d\epsilon, \\ 
y \in \bbR^n, 
\ X\in {\rm Mat}_{n\times k}, \ \rank X=k, 
\ \beta \in \bbR^k, 
\ \epsilon \in \bbR^n,  
\ \sigma>0,  
\end{gather*} 
where ${\rm Mat}_{n\times k}$ denotes the set of 
$n \times k$ matrices with real entries. 
Since $y \sim \sigma^{-n}f(v(\sigma^{-1}(y-X\beta)))dy$, 
we can deal with this model for $y$ with the same setting as that of 
Example  \ref{ex:v-setup}, 
but with $H$ and its action on $\calX=\{ y: y\in \bbR^n \}$ 
replaced by 
\[
H=\bbR^k, 
\quad 
H \times \calX \ni (\beta, y)\mapsto y+X\beta \in \calX. 
\] 
It is easy to see that the normalized residual vector 
\[  
t(y):=
\begin{cases} 
\frac{1}{\| e \| }e & \text{if \ $e \ne 0$,} \\ 
0 \in \bbR^n & \text{if \ $e = 0$,} 
\end{cases} 
\qquad 
e:=(I_n-X(X^TX)^{-1}X^T)y 
\] 
($I_n$: the $n\times n$ identity matrix), 
is invariant under the actions of $G=\bbR_{>0}$ and $H=\bbR^k$: 
\[
t(gy)=t(y), \ g>0, \quad t(y+X\beta)=t(y), \ \beta \in \bbR^k. 
\]
By Corollary \ref{co:inv-under-GandH}, the distribution of 
$t(y)=(1/\| e\|)e$ is the same for all 
$f$.  
In the elliptical case, 
a similar argument can be found in Section 3 of King \cite{King80}. 

\end{example} 

\section{Marginal equivalence} 
\label{sec:ME}

In this section, we establish marginal equivalence 
when the prior for $[g]$ is relatively invariant. 

Consider a (proper or improper) prior 
on $\Theta=H/H_0 \times G/G_0$ of the form 
\begin{equation} 
\label{eq:Pi=Pi*mu}
\Pi_{H/H_0} \otimes \tilde{\mu}_{G/G_0}. 
\end{equation} 
In \eqref{eq:Pi=Pi*mu}, 
$\Pi_{H/H_0}$ is a 
(proper or improper) prior on $H/H_0$, 
and $\tilde{\mu}_{G/G_0}:=\pi (m\mu_G)$ 
is a relatively invariant measure 
on $G/G_0$ with multiplier $m$: 
\[ 
\tilde{\mu}_{G/G_0}(d(g'[g])) 
=\tilde{\mu}_{G/G_0}(d[g'g])
=m(g')\tilde{\mu}_{G/G_0}(d[g]), \quad g' \in G, 
\] 
where $\pi: G \to G/G_0$ is the coset projection 
(Wijsman \cite{Wij90}, Corollary 7.4.4). 
The measure $\tilde{\mu}_G:=m\mu_G$ is a left invariant measure on $G$ 
if $m= 1$ and 
a right invariant measure on $G$ if $m=1/\Delta_G$.  

Let us consider the Bayesian model 
\begin{equation}
\label{eq:BM}
\left( p(x | [h], [g]; f), \ \Pi_{H/H_0} \otimes \tilde{\mu}_{G/G_0}\right) 
\end{equation}
with sampling model $p(x | [h], [g]; f)$ 
in \eqref{eq:p(x|hgf)} and 
prior 
$
\Pi_{H/H_0} \otimes \tilde{\mu}_{G/G_0}$. 
The density kernel 
\begin{equation} 
\label{eq:def-of-density-kernel}
p(x, [h] | f):=\int_{G/G_0} p(x | [h], [g]; f) \tilde{\mu}_{G/G_0}(d[g])
\end{equation} 
of $(x, [h])$ 
with respect to 
$\lambda \otimes \Pi_{H/H_0}$ 
is obtained in the following theorem. 

\begin{theorem} 
\label{th:p(x,h)}
For the Bayesian model 
\eqref{eq:BM}, 
assume 
$\int_G \chi_G(g^{-1})f(g^{-1})\tilde{\mu}_G(dg)<\infty$, 
or equivalently 
$\int_G \chi_G(g)f(g) \{ m(g) \Delta_G(g)\}^{-1}\mu_G(dg)<\infty$.  
Then the density kernel 
$p(x, [h] | f)$ of $(x, [h])$ 
with respect to 
$\lambda \otimes \Pi_{H/H_0}$ 
is proportional to $p(x, [h])$ defined by 
\begin{equation} 
\label{eq:density-kernel-of-(x,[h])-general-result} 
p(x, [h])
:=\frac{1}
{\chi_H(h) \tilde{\chi}_G(r(h^{-1}x))\tilde{\chi}_G(s(z(h^{-1}x)))}, 
\end{equation} 
where $\tilde{\chi}_G:=\chi_G/m$. 
\end{theorem} 

Note that $\chi_H(h)$ depends on $h\in H$ only through 
$[h]$ because $H_0$ is compact. 
Also, recall \eqref{eq:r(h0x)=r(x)-s(z(h0x))=s(z(x))} 
so that each of $\tilde{\chi}_G(r(h^{-1}x))$ and 
$\tilde{\chi}_G(s(z(h^{-1}x)))$ depends on $h\in H$ only 
through $[h]$.  

If $\Pi_{H/H_0}$ has a density $p_{H/H_0}$ with 
respect to a dominating measure $\lambda_{H/H_0}$, 
the density kernel of $(x, [h])$ with respect to 
$\lambda \otimes \lambda_{H/H_0}$ is 
proportional to the right-hand side of 
\eqref{eq:density-kernel-of-(x,[h])-general-result} 
multiplied by $p_{H/H_0}([h])$. 

\begin{proof} 
Since $\tilde{\mu}_{G/G_0}$ is the induced measure 
$\tilde{\mu}_{G/G_0}=\pi(\tilde{\mu}_G)$, 
we can express $p(x, [h] | f)$ 
in \eqref{eq:def-of-density-kernel} as 
\begin{eqnarray*}  
p(x, [h] | f)
&=& \int_G p(x | [h], [g]; f)\tilde{\mu}_G(dg) \\ 
&=& 
\frac{1}{\chi_H(h)}
\int_G 
\frac{1}{\chi_G(g)} 
f(g^{-1}r(h^{-1}x)s(z(h^{-1}x)))\tilde{\mu}_G(dg). 
\end{eqnarray*}  
Writing 
$r'=r'(x, [h]):=r(h^{-1}x)$ 
and 
$s'=s'(x, [h]):=s(z(h^{-1}x))$ 
(recall \eqref{eq:r(h0x)=r(x)-s(z(h0x))=s(z(x))}), 
we can calculate 
this integral 
as 
\begin{eqnarray*} 
p(x, [h] | f) 
&=& \frac{1}{\chi_H(h)}
\int_G 
\chi_G(g^{-1}) 
f(g^{-1}r's') \tilde{\mu}_G(dg) \\ 
&=& \frac{m(r's')}{\chi_H(h) \chi_G(r's')} 
\int_G 
\chi_G((g')^{-1})
f((g')^{-1})\tilde{\mu}_G(dg'), 
\quad g'=(r's')^{-1}g. 
\end{eqnarray*} 
\end{proof} 

\begin{rem}
Since $f$ is assumed to satisfy \eqref{eq:int-f-l=1}, 
we have $\int_G f(r) \chi_G(r) \mu_G(dr)<\infty$. 
Hence, if 
i) $\Delta_G=1, \ m=1$ ($G$ is unimodular and $\tilde{\mu}_G$ is 
invariant) 
or more generally 
ii) $m\Delta_G=1$ ($\tilde{\mu}_G$ is right invariant), 
then the integrability condition in Theorem \ref{th:p(x,h)} 
is automatically satisfied.  
\end{rem} 

Recall that $[h]$ is the parameter of interest. 
From Theorem \ref{th:p(x,h)}, we see that 
for fixed $r: \calX_* \to G$ satisfying \eqref{eq:rgx=grx} 
and $s: \calZ \to G$, 
the Bayesian models 
\eqref{eq:BM} 
with different $f$s 
lead to the same 
$p(x, [h])$, 
and thus are all marginally equivalent 
in the sense of 
\cite{OS93JE59} 
and 
\cite{FOS95JASA}. 
Note that 
$\lambda \otimes \Pi_{H/H_0}$ 
does not depend on $f$. 

Suppose, moreover, that 
\[
p(x):=\int_{H/H_0} p(x, [h])\Pi_{H/H_0}(d[h])<\infty. 
\] 
Then the posterior $p([h] | x):=p(x, [h])/p(x)$ is 
proper: $\int_{H/H_0} p([h] | x)\Pi_{H/H_0}(d[h])=1$.  
In that case, 
$p(x)$ and $p([h] | x)$, like $p(x, [h])$,   
do not depend on $f$. 

\begin{example} 
\label{ex:v-ME} 
{\bf (Continuing Examples \ref{ex:v-setup} and \ref{ex:v-null}.)} 
Recall $\mu_G(dg)=\mu_{ \bbR_{>0}}(dg)
=g^{-1}dg$. 
For an arbitrary prior $\Pi_{H}$ on $H=\bbR^n$, 
we have from Theorem \ref{th:p(x,h)} that 
for the Bayesian model 
\[ 
\left( p(x | \mu, \sigma; f), \ \Pi_{\bbR^n} \otimes m \mu_{\bbR_{>0}}\right),  
\]  
the density kernel of 
$(x, \mu)$ with respect to $\lambda \otimes 
\Pi_{\bbR^n}$ 
($\lambda$: Lebesgue measure) 
is proportional to 
\[
p(x, \mu)=\frac{m(v(x-\mu))}{ \{ v(x-\mu)\}^{n}}. 
\] 
When $m=1$ (i.e., the prior on $G=\bbR_{>0}$ is noninformative), 
this reduces to 
Theorem 1 of Fern\'{a}ndez, Osiewalski and Steel \cite{FOS95JASA}.  

\end{example} 

\section{Examples} 
\label{sec:examples} 

We apply our general results to two examples other than the 
$v$-spherical distribution. 
We consider affine shapes in Subsection \ref{sec:affine-shapes} 
and PCA in Subsection \ref{sec:PCA}.  

\subsection{Affine shapes} 
\label{sec:affine-shapes}

We will consider null robustness (central configuration density)  
and marginal equivalence 
in the analysis of affine shapes. 

Let a real $N \times k$ matrix 
\[
X=
\begin{pmatrix} 
x_1^T \\ 
\vdots \\ 
x_N^T 
\end{pmatrix} 
\in {\rm Mat}_{N\times k}
\]
represent a geometric figure consisting of 
$N$ landmark points 
$x_1,\ldots,x_N$ in $\bbR^k$. 
Two figures 
$X \in {\rm Mat}_{N\times k}$ 
and 
$X' \in {\rm Mat}_{N\times k}$ 
are said to have 
the same affine shape (or the same configuration) iff 
\[
X=X' E^T + \iota_Nb^T
\] 
for some 
$b\in \bbR^k$ and 
nonsingular $E \in {\rm Mat}_{k \times k}$, 
where $\iota_N=(1,\ldots,1)^T\in \bbR^N$ 
(Goodall and Mardia \cite[Section 6]{GM93}, 
Caro-Lopera, D\'{i}az-Garc\'{i}a and Gonz\'{a}lez-Far\'{i}as 
\cite{CDG10}). 
Now, let $L \in {\rm Mat}_{n\times N}$ 
($n:=N-1$) 
be the submatrix of the Helmert matrix, 
consisting of $n$ orthonormal rows 
which are orthogonal to $\iota_N^T$. 
Suppose $n-k\ge 1$. 
Then the configuration coordinates of a figure 
$X \in {\rm Mat}_{N\times k}$ are 
given by 
\[
\begin{pmatrix} 
I_k \\ 
V
\end{pmatrix}
\in {\rm Mat}_{n\times k}, 
\] 
where 
\[ 
V:=Y_2Y_1^{-1}
\] 
with 
\[
Y=
\begin{pmatrix}
Y_1 \\ 
Y_2 
\end{pmatrix}
:=LX \in {\rm Mat}_{n\times k}, 
\quad  
Y_1\in {\rm Mat}_{k\times k}, 
\ 
Y_2 \in {\rm Mat}_{(n-k)\times k} 
\] 
when 
$Y_1$ is nonsingular 
(Goodall and Mardia \cite[Section 6]{GM93}, 
Caro-Lopera, D\'{i}az-Garc\'{i}a and Gonz\'{a}lez-Far\'{i}as 
\cite{CDG10}). 

We assume $Y$ is distributed with density 
\begin{equation} 
\label{eq:affine-dist-of-Y} 
\frac{1}{(\det \Sigma_0)^{k/2} (\det \Phi)^{n/2}}
\tilde{f}((Y-M)^T\Sigma_0^{-1}(Y-M)\Phi^{-1}), 
\quad M\in {\rm Mat}_{n\times k}, \ \Phi \in {\rm PD}_k,  
\end{equation} 
with respect to Lebesgue measure, 
where $\Sigma_0 \in {\rm PD}_n$ is known and 
$\tilde{f}$ satisfies 
\begin{equation} 
\label{eq:fAB=fBA}
\tilde{f}(AB)=\tilde{f}(BA)
\end{equation}  
whenever $AB$ and $BA$ are defined 
(e.g., $\tilde{f}( \, \cdot \, )=h(\tr ( \, \cdot \, ))$ 
for some $h( \, \cdot \, )$).  
When 
$\Phi =I_k$ and $\tilde{f}( \, \cdot \, )=h(\tr ( \, \cdot \, ))$, 
this assumption about 
the distribution of $Y$ is the same as 
that of 
Caro-Lopera, D\'{i}az-Garc\'{i}a and Gonz\'{a}lez-Far\'{i}as 
\cite{CDG10}. 

As we will see, we can regard 
the configuration coordinates 
$(I_k, V^T)^T = (I_k, (Y_2 Y_1^{-1})^T)^T$ 
as the maximal invariant 
$z(x)$ in \eqref{eq:def-of-z(x)} 
and the density \eqref{eq:affine-dist-of-Y} 
as $p(x | h, g; f)$ in \eqref{eq:p(x|hgf)}. 

Let 
\[
\calX={\rm Mat}_{n\times k}, 
\quad 
\calX_*
=
\left\{ 
\begin{pmatrix}
Y_1 \\ 
Y_2
\end{pmatrix}: 
Y_1 \in {\rm Mat}_{k \times k}, \ 
\rank (Y_1)=k, \ 
Y_2 \in {\rm Mat}_{(n-k)\times k}
\right\}
\]  
and 
\begin{eqnarray*} 
H &=& {\rm Mat}_{n\times k} 
\ (\text{the additive group under matrix addition}), \\  
G &=& {\rm GL}_{k} \ (\text{the real general linear group}).  
\end{eqnarray*} 
Note $\Delta_G=1$.  
Consider the actions 
\[ 
H\times \calX \ni (M, Y)\mapsto Y+M \in \calX,  
\qquad 
G\times \calX_* \ni (E, Y)\mapsto YE^T \in \calX_*. 
\] 
We take 
$\lambda(dY)=(dY)$ (Lebesgue measure)  
so that 
$\chi_H(M)=1$ 
and 
$\chi_G(E)=|\det E \, |^{n}$. 

Let 
\[ 
r(Y):=Y_1^T
\ \text{ for } \ 
Y=
\begin{pmatrix}
Y_1 \\ 
Y_2 
\end{pmatrix}. 
\] 
Then 
\[
z(Y)=
\begin{pmatrix}
I_k \\ 
Y_2 Y_1^{-1}
\end{pmatrix} 
\ \text{ for } \ 
Y=
\begin{pmatrix}
Y_1 \\ 
Y_2 
\end{pmatrix} 
\]
and 
\[ 
\calZ=
\left\{ 
Z= 
\begin{pmatrix}
I_k \\ 
V
\end{pmatrix}: 
V \in {\rm Mat}_{(n-k)\times k}
\right\}. 
\] 
Thus the matrix 
$(I_k, V^T)^T = (I_k, (Y_2 Y_1^{-1})^T)^T$ 
consisting of the configuration coordinates 
is the maximal invariant $z(x)$ in \eqref{eq:def-of-z(x)}. 


Next, let us check that \eqref{eq:affine-dist-of-Y} can be 
regarded as $p(x | h, g; f)$  
in \eqref{eq:p(x|hgf)}. 
Let $S^{1/2}$ for $S \in {\rm PD}_k$ 
stand for the unique $T \in {\rm LT}_{k}$ 
(the group of  
$k\times k$ lower triangular 
matrices with 
positive diagonal elements) 
such that $S=TT^T$. 
Putting 
\[
s(Z):=(Z^T\Sigma_0^{-1}Z)^{1/2}
\in {\rm LT}_{k} \subset {\rm GL}_k=G, 
\quad 
Z \in \calZ, 
\]
and  
$f(E):=(\det \Sigma_0)^{-k/2}\tilde{f}(EE^T), 
\ E \in {\rm GL}_k$,  
we can express \eqref{eq:affine-dist-of-Y} as 
\begin{equation} 
\label{eq:affine-density-f=p} 
\frac{1}{(\det \Phi)^{n/2}}
f( \Phi^{-1/2} r(Y-M) s(z( Y-M  )))
=p(Y | M, \Phi^{1/2}; f)
\end{equation}  
in \eqref{eq:p(x|hgf)}, 
where 
$\Phi^{-1/2}:=(\Phi^{1/2})^{-1}$ 
(recall \eqref{eq:fAB=fBA}). 
For any $C \in {\rm O}_k$ (the orthogonal group),  
we have $f(CE)=f(E)$ (use \eqref{eq:fAB=fBA} again), 
so $G_0={\rm O}_k$ and we can write 
\eqref{eq:affine-density-f=p} as 
$p(Y | M, \Phi^{1/2}{\rm O}_k; f)$. 

The distribution of the configuration coordinates 
in the central case $M=0$ is obtained 
as follows. 
When $M=0$, 
we have from Theorem \ref{th:null-robust} that 
the density of $Z=z(Y)$ is 
$p_{\calZ}(Z | 0, E{\rm O}_k; f)=p_{\calZ}(Z)
=c \, \{ \det (Z^T\Sigma_0^{-1}Z)\}^{-n/2}
$ 
for any $E{\rm O}_k \in {\rm GL}_k/{\rm O}_k$ 
so 
\[
Z 
\sim 
\frac{c}{\{ \det (Z^T\Sigma_0^{-1}Z) \}^{n/2}} 
\nu_{\calZ}(dZ). 
\]
Since 
\[
(dY)=|\det E \, |^{n-k}(dE) \, (dV) 
\ \text{ for }  
Y=
\begin{pmatrix} 
I_k \\ 
V
\end{pmatrix} 
E^T, 
\] 
we see that  
\[ 
\nu_{\calZ}(dZ)=(dV) \ \text{ for } \ 
Z=
\begin{pmatrix} 
I_k \\ 
V
\end{pmatrix} 
\] 
when we take the version 
\begin{equation} 
\label{eq:Haar-GLk}
\mu_G(dE)=\mu_{{\rm GL}_k}(dE)
=|\det E \, |^{-k}(dE).  
\end{equation} 
By considering the particular case of normality: 
$f(E)=(2\pi)^{-kn/2}(\det \Sigma_0)^{-k/2}{\rm etr}\{ -(1/2)EE^T\}$, 
i.e., $\tilde{f}(W)=(2\pi)^{-kn/2}{\rm etr}\{ -(1/2)W \}$ 
(Muirhead \cite{Muir82}, Theorem 3.1.1), 
we can easily obtain the normalizing constant as 
\begin{eqnarray*}
c
&=& \int_{{\rm GL}_k} f(E) \, |\det E \, |^{n} \mu_{{\rm GL_k}}(dE) \\ 
&=& 
\frac{1}{(2\pi)^{kn/2}(\det \Sigma_0)^{k/2}} 
\, \mu_{{\rm O}_k}({\rm O}_k) 
\int_{{\rm LT}_k}{\rm etr}\left(-\frac{1}{2}TT^T\right)  
\{ \det (TT^T) \}^{n/2}
\mu_{{\rm LT}_k}(dT) \\ 
&=& 
\frac{1}{(2\pi)^{kn/2}(\det \Sigma_0)^{k/2}} 
\frac{2^k \pi^{k^2/2}}{\Gamma_k(\frac{k}{2})}  
\int_{{\rm PD}_k}{\rm etr}\left( -\frac{1}{2}W \right)  (\det W)^{n/2}
2^{-k} (\det W)^{-(k+1)/2}
 (dW) \\ 
&=& \frac{\Gamma_k(\frac{n}{2})}
{\pi^{k(n-k)/2}(\det \Sigma_0 )^{k/2}\Gamma_k(\frac{k}{2})} 
\end{eqnarray*} 
for the version of $\mu_G$ in \eqref{eq:Haar-GLk}, 
where 
$\mu_{{\rm LT}_k}(dT)=\prod_{i=1}^k t_{ii}^{-i}(dT), \ T=(t_{ij})\in {\rm LT}_k$, 
the version of $\mu_{{\rm O}_k}$ is such that $\mu_{{\rm O}_k}(dC)=(dC)$ at 
$C =I_k$, 
and 
$\allowbreak 
\Gamma_k(a)=\int_{{\rm PD}_k} {\rm etr}(-W) (\det W)^{a-\{ (k+1)/2\}}(dW)
=\pi^{k(k-1)/4} \prod_{i=1}^k \Gamma (a-\{ (i-1)/2\}), 
\ {\rm Re}(a)>(k-1)/2$. 
(We have used Wijsman \cite{Wij90}, (7.7.10), (7.7.9), (5.3.16) and 
Muirhead \cite{Muir82}, Definition 2.1.10, Theorem 2.1.12.)
Hence 
\[
V \sim 
\frac{\Gamma_k(\frac{n}{2})}
{\pi^{k(n-k)/2}(\det \Sigma_0 )^{k/2}\Gamma_k(\frac{k}{2})} 
\left[ 
\det 
\left\{ 
(I_k, V^T)
\Sigma_0^{-1}
\begin{pmatrix} 
I_k \\ 
V
\end{pmatrix} 
\right\}
\right]^{-n/2} 
(dV)  
\] 
in the central case $M=0$. 
When $\Phi =I_k$ and 
$\tilde{f}( \, \cdot \, )=h(\tr ( \, \cdot \, ))$, 
this agrees with Corollary 10 of 
Caro-Lopera, D\'{i}az-Garc\'{i}a and Gonz\'{a}lez-Far\'{i}as 
\cite{CDG10}. 

Next we move on to marginal equivalence, 
supposing that we are interested in $M$ but not in $\Phi$. 
Let us take $m=1$.  
Recall \eqref{eq:Haar-GLk}. 
For an arbitrary prior $\Pi_H$ on $H={\rm Mat}_{n \times k}$, 
we have from Theorem \ref{th:p(x,h)} that 
for the Bayesian model 
\[ 
\left( p(Y | M, \Phi^{1/2}{\rm O}_k; f), \ 
\Pi_{{\rm Mat}_{n\times k}} \otimes \mu_{{\rm GL}_k/{\rm O}_k}\right) 
\]  
with the invariant measure 
$\mu_{{\rm GL}_k/{\rm O}_k}=\pi(\mu_{{\rm GL}_k})$ on 
${\rm GL}_k/{\rm O}_k$, 
the density kernel of 
$(Y, M)$ with respect to 
$\lambda \otimes  
\Pi_{{\rm Mat}_{n\times k}}$ 
($\lambda$: Lebesgue measure) 
is proportional to 
\begin{eqnarray*} 
p(Y, M)
&=& \left| \det (Y_1-M_1) \right|^{-n} \\ 
&& \times 
\left[ \det \left\{ 
\left( 
I_k, (Y_1-M_1)^{-T}(Y_2-M_2)^T
\right) 
\Sigma_0^{-1}
\begin{pmatrix} 
I_k \\ 
(Y_2-M_2)(Y_1-M_1)^{-1}
\end{pmatrix} 
\right\} \right]^{-n/2}
\end{eqnarray*} 
for $Y=(Y_1^T, Y_2^T)^T$ and $M=(M_1^T, M_2^T)^T$, 
where $(Y_1-M_1)^{-T}:=((Y_1-M_1)^{-1})^T$. 

\subsection{Marginal equivalence in PCA} 
\label{sec:PCA} 

We will consider marginal equivalence in PCA 
when the eigenvalues are known up to 
a positive scalar multiplication. 

Suppose a random matrix 
$X \in {\rm Mat}_{n \times k}$ is distributed as 
\begin{equation} 
\label{eq:dist-of-X-in-PCA}  
X \sim \frac{1}{(\det \Sigma)^{n/2}} 
\tilde{f}(\tr (X\Sigma^{-1}X^T)) 
(dX), 
\end{equation} 
where 
\[ 
\Sigma = g^2 P\Lambda_0 P^T 
\] 
with 
$g>0, \ P \in {\rm O}_k, \ \Lambda_0 :=\diag (\ell_{1, 0},\ldots,\ell_{k, 0})$ 
($\ell_{1, 0}>\cdots >\ell_{k, 0}>0$). 
We assume $\Lambda_0$ is known. 
The parameter of interest is $P$ 
(ignoring the sign of each column), 
and $g$ is a nuisance parameter. 

Let 
\[ 
\calX = {\rm Mat}_{n\times k}, 
\quad 
\calX_* = \{ X=(x_1,\ldots,x_k) \in {\rm Mat}_{n\times k}: 
x_i \neq 0 \in \bbR^n \text{ for some } i=1,\ldots,k \}
\] 
and 
\[ 
H={\rm O}_k, \quad G=\bbR_{>0}
\]  
so that $\Delta_G=1$. 
Consider the actions 
\begin{eqnarray*} 
&& H \times \calX \ni 
(P, X) \mapsto XP^T \in \calX, \\ 
&& G \times \calX_* \ni 
(g, X) \mapsto gX \in \calX_* \text{ (scalar multiplication)}. 
\end{eqnarray*} 

We define 
\[
r(X):=\left\{ \tr ( X\Lambda_0^{-1}X^T)\right\}^{1/2}>0, \quad X \in \calX_*. 
\] 
Moreover, we take $\lambda(dX)=(dX)$ (Lebesgue measure), so that 
$\chi_H(P)=1, \ \chi_G(g)=g^{kn}$.  
Then the distribution of $X$ in \eqref{eq:dist-of-X-in-PCA} 
can be written as 
\begin{eqnarray*} 
&& \frac{1}{\chi_G(g) (\det \Lambda_0)^{n/2} } 
\tilde{f}(\{ r(g^{-1}X P^{-T})\}^2)
\lambda(dX) \\ 
&& \quad = \frac{1}{\chi_G(g)}f(r(g^{-1}X P^{-T})) \lambda(dX)  
= p(X | P, g; f) \lambda(dX) 
\end{eqnarray*} 
with  
$f( r ):=(\det \Lambda_0)^{-n/2}\tilde{f}( r^2 )$ 
and $s( \, \cdot \, )=1$. 

For any 
$P_0=\diag (\epsilon_1, \ldots, \epsilon_k) 
\ (\epsilon_i=\pm 1, \ i=1,\ldots,k)$, 
we have $r(XP_0)=r(X)$. 
Thus we can take 
\[
H_0=\{ P_0=\diag (\epsilon_1, \ldots, \epsilon_k):  
\epsilon_i=\pm 1, \ i=1,\ldots,k \}, 
\quad 
G_0=\{ 1 \}. 
\]
For brevity's sake, we will write the $H_0$ above as $\{ P_0 \}$. 

Now, concerning marginal equivalence, we obtain the following 
result. 
Let us take $m=1$. 
For an arbitrary prior $\Pi_{H/H_0}$ on $H/H_0={\rm O}_k/\{ P_0 \}$, 
we have from Theorem \ref{th:p(x,h)} that 
for the Bayesian model 
\[ 
\left( p(X | P\{ P_0 \}, g; f), \ \Pi_{{\rm O}_k/\{ P_0 \}}\otimes \mu_{\bbR_{>0}}\right) 
\]  
with noninformative $\mu_{\bbR_{>0}}(dg)=g^{-1}dg$, 
the density kernel of 
$(X, [P])=(X, P\{ P_0 \})$ 
with respect to $\lambda \otimes 
\Pi_{{\rm O}_k/\{ P_0 \}}$ 
($\lambda$: Lebesgue measure) 
is proportional to 
\[
p(X, [P])=\left\{ \tr (XP\Lambda_0^{-1}P^TX^T) \right\}^{-kn/2}. 
\] 


\section{Concluding remarks} 
\label{sec:concluding-remarks} 

In the framework of group invariance, 
we introduced a class of sampling models 
which can be regarded as a generalization of the $D$-class 
of distributions of \cite{FS05JRSS(B)}. 
Under these sampling models, we obtained 
some robustness properties of invariant statistics. 
When the prior for the nuisance parameter is 
relatively invariant, we also derived marginal equivalence. 
However, when the prior for the nuisance parameter is not 
relatively invariant, results about marginal equivalence 
as general as those in this paper do not seem to hold 
and only more restricted results can be expected 
(see Subsection 5.2 of \cite{FOS95JASA} for the case of $v$-spherical distributions). 

%
%

\end{document}